\documentclass[12pt,reqno]{amsart}
\newtheorem{theorem}{Theorem}[section]

\begin{document}

\baselineskip=15.5pt

\title[Automorphisms of $\overline{T}$]{Automorphisms of $\overline{T}$}

\author[I. Biswas]{Indranil Biswas}

\address{School of Mathematics, Tata Institute of Fundamental
Research, Homi Bhabha Road, Mumbai 400005, India}

\email{indranil@math.tifr.res.in}

\author[S. S. Kannan]{S. Senthamarai Kannan}

\address{Chennai Mathematical Institute, H1, SIPCOT IT Park, Siruseri,
Kelambakkam 603103, India}

\email{kannan@cmi.ac.in}

\author[D. S. Nagaraj]{D. S. Nagaraj}

\address{The Institute of Mathematical Sciences, CIT
Campus, Taramani, Chennai 600113, India}

\email{dsn@imsc.res.in}



\begin{abstract}
Let $\overline G$ be the wonderful compactification of 
a simple affine algebraic group $G$ defined over $\mathbb C$ such that
its center is trivial and $G \,\not=\, {\rm PSL}(2,\mathbb{C})$. Take a maximal
torus $T\, \subset\, G$, and denote by $\overline T$ its closure in $\overline G$.
We prove that $T$ coincides with the connected component, containing the identity
element, of the group of automorphisms of the variety $\overline T$.
\medskip

\textbf{R\'esum\'e}. \textit{Automorphismes de $\overline{T}$}\\
Soit $\overline G$ la compactification magnifique d'un groupe alg\'ebrique affine $G$
d\'efini sur $\mathbb C$, dont le centre est trivial et tel que
$G \,\not=\, {\rm PSL}(2,\mathbb{C})$. Soit $T\, \subset\, G$ un tore maximal, et
soit $\overline T$ son adh\'erence dans $\overline G$. Nous montrons que $T$ est \'egal
\`a la composante connexe contenant l'\'el\'ement neutre du groupe d'automorphismes de
la vari\'et\'e $\overline T$.
\end{abstract}

\maketitle

\section{Introduction}

Let $G$ be a simple affine algebraic group defined over the complex numbers such that
the center of $G$ is trivial. De Concini and Procesi
constructed a very interesting compactification
of $G$ which is known as the wonderful compactification 
\cite[p. 14, 3.1, THEOREM]{DP}. The wonderful
compactification of $G$ will be denoted by $\overline G$. Fix a maximal torus $T$ of
$G$. Let $\overline T$ denote the closure of $T$ in $\overline G$. The connected
component, containing the identity element, of the group of all automorphisms of the
variety $\overline T$ will be denoted by ${\rm Aut}^0(\overline{T}).$ 
For more details about the  variety $\overline{T}$ we refer to 
\cite[\S~1]{BJ}. Our aim here is to compute ${\rm Aut}^0(\overline{T})$.

Using the action of $G$ on $\overline G$, we have $T\, \subset\,
{\rm Aut}^0(\overline{T})$; this inclusion does not depend on whether the right or
the left action is chosen. We prove that $T\, =\,
{\rm Aut}^0(\overline{T})$, provided $G \,\not=\, {\rm PSL}(2,\mathbb{C})$; see Theorem
\ref{thm1}.

Note that $Aut(\overline{T})$ is not connected since
$\overline{T}$ is stable under the conjugation of the normalizer $N_{G}(T)$ of $T$ in $G.$

If $G\,=\, {\rm PSL}(2,\mathbb{C})$, then $\overline{T}\, =\, {\mathbb P}^1$,
and hence ${\rm Aut}^0(\overline{T})\,=\, {\rm PSL}(2,\mathbb{C})$.

\section{Lie algebra and algebraic groups}

In this section we recall some basic facts and notation on Lie algebra and algebraic
groups (see \cite{Hu}, \cite{Hu1} for details). Throughout $G$ denotes an affine
algebraic group over $\mathbb{C}$ which is simple and of adjoint type. We also assume
that the rank of $G$ is at-least two, equivalently $G \,\not=\, {\rm PSL}(2,\mathbb{C})$.

For a maximal torus $T$ of $G$, the group of all characters
of $T$ will be denoted by $X(T)$. The Weyl group of $G$ with respect to $T$
is defined to be $W\ :=\, N_{G}(T)/T$, where $N_{G}(T)$ is the normalizer of $T$ in $G$.
By $R \,\subset \,X(T)$ we denote the root system of $G$ with respect to $T$.
For a Borel subgroup $B$ of $G$ containing $T$, let 
$R^{+}(B)$ denote the set of positive roots determined by $T$ and $B$. Let
$$
S \,=\, \{\alpha_1\, ,\cdots\, ,\alpha_n\}
$$ 
be the set of simple roots in $R^{+}(B).$ Let $B^{-}$ denote the opposite
Borel subgroup of $G$ determined by $B$ and $T.$ For $\alpha \,\in\, R^{+}(B)$, let 
$s_{\alpha} \,\in \, W$ be the 
reflection corresponding to $\alpha$. The Lie algebras of
$G$, $T$ and $B$ will be denoted by $\mathfrak{g}$, $\mathfrak{t}$ and $\mathfrak{b}$
respectively. The dual of the real form $\mathfrak{t}_{\mathbb R}$ of $\mathfrak{t}$ is
$X(T)\otimes \mathbb{R}\,=\, Hom_{\mathbb{R}}(\mathfrak{t}_{\mathbb{R}},\, \mathbb{R})$.
 
The positive definite $W$--invariant form on $Hom_{\mathbb{R}}(\mathfrak{t}_{\mathbb{R}},
\, \mathbb{R})$ induced by the Killing form on $\mathfrak{g}$ is denoted by $(~,~)$. 
We use the notation $$\langle \nu\, , \alpha \rangle \,:= \,\frac{2(\nu,
\alpha)}{(\alpha,\alpha)}\, .$$ In this setting one has the Chevalley basis 
$$\{x_{\alpha}, h_{\beta} \,\mid\, \alpha \,\in\,
R,\, ~ \beta \,\in\, S\}$$ of $\mathfrak{g}$ determined by $T$.
For a root $\alpha$, we denote by 
$U_{\alpha}$ (respectively, $\mathfrak{g}_{\alpha}$) the one--dimensional
$T$ stable root subgroup of $G$ (respectively, the 
subspace of $\mathfrak{g}$) on which $T$ acts through the 
character $\alpha$.

Now, let $\sigma$ be the involution of $G\times G$ defined by 
$\sigma(x\, ,y)\,=\,(y\, ,x)$. Note that the diagonal subgroup $\Delta(G)$
of $G\times G$ is the subgroup of fixed points, while
$T\times T$ is a $\sigma$-stable maximal torus of $G\times G$ and
$B\times B^{-}$ is a Borel subgroup having the property that 
$\sigma(\alpha)\,\in\, -R^{+}(B\times B^{-})$
for every $\alpha \,\in\, R^{+}(B\times B^{-}).$

Let $\overline{G}$ denote the wonderful compactification of the group $G$, 
where $G$ is identified with the symmetric space $(G\times G)/\Delta(G)$ 
(see \cite[p. 14, 3.1. THEOREM]{DP}).
Let $\overline{T}$ be the closure of $T$ in $\overline{G}.$

\section{The connected component of the automorphism group}

Recall that if $X$ is a smooth projective variety over $\mathbb{C},$ the
connected component of the group of all automorphisms of $X$ containing the
identity automorphism is an algebraic group (see \cite[p. 17, Theorem 3.7]{MO}
and \cite[p. 268]{Gr}), which deal also the case when $X$ may be singular or
it may be defined over any field).  Further, the Lie algebra  of this
automorphism group is isomorphic to the space of all vector fields on $X,$
that is the space $H^{0}(X, \Theta_{X})$ of all global sections of the tangent
 bundle $\Theta_{X}$ of $X$ ( see \cite[p. 13, Lemma 3.4]{MO}). 

Let ${\rm Aut}(\overline{T})$ denote the group of all algebraic automorphisms
of the variety $\overline{T}$. Let
$${\rm Aut}^0(\overline{T})\,\subset\, {\rm Aut}(\overline{T})$$
be the connected component containing the identity element. We note that
${\rm Aut}^0(\overline{T})$ is an algebraic group with Lie algebra
$\text{H}^0(\overline{T},\, \Theta_{\overline{T}}),$ where $\Theta_{\overline{T}}$
is the tangent bundle of the variety $\overline{T}$; the Lie algebra structure on
$\text{H}^0(\overline{T},\, \Theta_{\overline{T}})$ is given by the Lie bracket
of vector fields.

The subvariety $\overline{T}\, \subset\, \overline{G}$ is stable under the action of 
$T\times T.$ Further, the subgroup $T\times 1\,\subset\, T\times T$ acts faithfully on
$\overline{T}$, and $T \,\subset\, \overline{T}$ is a stable Zariski open dense subset for
this action of $T.$ Hence, we get an injective homomorphism
$$
\rho\, :\, T \,\longrightarrow\, {\rm Aut}^0(\overline{T})\, .
$$

\begin{theorem}\label{thm1}
The above homomorphism $\rho$ is an isomorphism.
\end{theorem}

\begin{proof}
We know that $T$ is a maximal torus of ${\rm Aut}^0(\overline{T})$ 
\cite[p. 521, COROLLAIRE 1]{De}. Choose
a Borel subgroup $B^{\prime}\, \subset\, {\rm Aut}^0(\overline{T})$ containing
the maximal torus $T$ of ${\rm Aut}^0(\overline{T})$.
The action of $B^{\prime}$ on $\overline{T}$ fixes a point in because
$\overline{T}$ is a projective variety (see 
\cite[p. 134, 21.2, Theorem]{Hu1}). Let $x\,\in\,
\overline{T}$ be a point fixed by $B^{\prime}.$ Clearly,
$n\overline{T}n^{-1}\,=\, \overline{T}$ for $n\,\in\, N_G(T)$, and the diagonal 
subgroup of $T\times T$ acts trivially on $\overline{T}$. Hence $W\,=\, N_{G}(T)/T$
is a subgroup of ${\rm Aut}(\overline{T}).$
The diagonal subgroup of $T\times T$ acts trivially on
$\overline{T}$. So we see that $T\times T$ fixes the point $x.$
Therefore, by \cite[p. 477, (1.2.7)]{BJ} and \cite[p. 478, (1.3.8)]{BJ} we have
that $x\,=\,w(z)$ for some
$w\,\in\, W,$ where $z$ is the unique $B\times B^{-}$ fixed point in $\overline{G}.$
Using conjugation by $w^{-1}$, we may assume that $B^{\prime}$ fixes $z.$ Let $$Q\,\subset\,
{\rm Aut}^0(\overline{T})$$
be the stabilizer subgroup for the point $z.$ As $B^{\prime} \,\subset\, Q,$ it
follows that $Q$ is in fact a parabolic subgroup of 
${\rm Aut}^{0}(\overline{T}).$ 

We first show that ${\rm Aut}^{0}(\overline{T})$ is reductive.
Let $R_{u}$ be the unipotent radical of ${\rm Aut}^{0}(\overline{T}).$
Therefore, $R_{u}$ is also the unipotent radical of ${\rm Aut}(\overline{T}).$
Hence $wR_{u}w^{-1}\,=\,R_{u}$ for all $w\,\in\, W.$ Consequently, $R_{u}
\,\subset\, B^{\prime}$ fixes $w(z)$ for every $w\,\in\, W.$

For $\chi\,\in\, X(B)\,=\,X(T)$, let $\mathcal{L}_{\chi}$ be the line bundle on
$\overline{G}$ associated to $\chi$ (see \cite[p. 26, 8.1, PROPOSITION]{DP}). Take any $w\,\in \,W.$
The action of $R_{u}$ fixes $w(z)$, so
the fiber $(\mathcal{L}_{\chi})_{w(z)}$ of $\mathcal{L}_{\chi}$ 
over $w(z)$ is an one dimensional representation of $R_{u}$. This
$R_{u}$--module $(\mathcal{L}_{\chi})_{w(z)}$ is trivial because the group
$R_{u}$ is unipotent.

Let $\mathbb{C}[T]$ be the coordinate ring of the affine algebraic group
$T.$ We note that $\mathbb{C}[T]$ is a unique factorization domain, and
therefore any line bundle on $T$ is
trivial. As $T \,\subset\, \overline{T}$ is a $T$ stable open dense subset for the left
translation action, we see that 
 the $T$ module $H^{0}(\overline{T},
\, \mathcal{L}_{\chi})$ is a submodule of $\mathbb{C}[T].$ If $\chi$
is a dominant character of $T,$ and $w\,\in\, W$, then the weight space $\mathbb{C}[T]$ 
of weight $-w(\chi)$ is one dimensional 
and spanned by $t^{-w(\chi)}.$ Moreover, we have $t^{-w(\chi)} \,\in\, H^{0}(\overline{T},
\, \mathcal{L}_{\chi}),$ because it is the unique section of weight $-w(\chi)$
not vanishing at $w(z).$ Thus, from the above it follows that
$t^{-w(\chi)}$ is fixed by $R_{u}$ for every dominant character $\chi$ of $T$ and 
every $w\,\in\, W.$

The set $\{t^{\chi}\,\mid\, \chi\,\in\, X(T)\}$
is a basis for the complex vector space $\mathbb{C}[T].$ 
Therefore, the action of $R_{u}$ on $H^{0}(\overline{T},\, \mathcal{L}_{\chi})$ 
is trivial
for every regular dominant character $\chi$ of $T.$ We have 
$$\overline{T}\,\subset \,\mathbb{P}(H^{0}(\overline{T} , \,\mathcal{L}_{\chi}))
\, ,$$ 
and hence it follows that the action of $R_{u}$ on $\overline{T}$ is trivial,
implying that $R_u$ is trivial. Thus, 
the group ${\rm Aut^{0}}(\overline{T})$ is reductive. 

Next we will show that $Q\,=\,B^{\prime}.$ Fix a dominant character $\chi$ of
$T\,\subset\, B.$ As ${\rm Aut^{0}}(\overline{T})$ is reductive 
${\rm Aut^{0}}(\overline{T})/Z$ is semisimple, where $Z(\subset Q)$ is the center of ${\rm Aut^{0}}(\overline{T}).$ Note that $Q/Z$ is a parabolic subgroup of ${\rm Aut^{0}}(\overline{T})/Z$ and it fixes $z,$ by the 
arguments in  the proofs of  \cite[p. 81--82, 3.2, Proposition and 3.3,
Corollary]{KKV} there is a positive integer $a$ such that $Q/Z$ acts
linearly on the fiber of the line bundle 
$\mathcal{L}^{\otimes a}_{\chi}$ over $z$ through some character $\chi^{\prime}$ of $Q/Z.$ Pulling back $\chi^{\prime}$ to $Q$ we see that $Q$ acts on the fiber the line bundle $\mathcal{L}^{\otimes a}_{\chi}$ over $z$ by a character. 
The group $X(T)$ is finitely generated and Abelian, and hence the image of the
restriction map $$X(Q)\,\longrightarrow\, X(B^{\prime})\,=\,X(T)$$
is of finite index. This implies that the rank of $X(Q)$ is equal to the the rank of
$X(B^{\prime}).$ Thus, we have $Q\,=\,B^{\prime}.$

We will now show that ${\rm Aut}^{0}(\overline{T})$ is not semisimple.
If ${\rm Aut}^{0}(\overline{T})$ is semisimple, then 
\begin{equation}\label{dim}
\dim \overline{T}\,=\, \dim T\,\leq \,\dim B_{u}^{\prime}
\,=\,\dim ({\rm Aut}^0(\overline{T})/B^{\prime})\, ,
\end{equation}
where $B_{u}^{\prime}$ is the unipotent radical of $B^{\prime}.$ 
 Note that by the above observation,
$B^{\prime}$ is the stabilizer of $z$ in $Aut^{0}(\overline{T})$. Since $B^{\prime}$ is a Borel subgroup of $Aut^{0}(\overline{T})$,
${\rm Aut}^0(\overline{T})/B^{\prime}$ is a closed 
subvariety of $\overline{T}.$ Thus from \eqref{dim} we get that
${\rm Aut}^0(\overline{T})/B^{\prime} \,=\,\overline{T}.$ 
This implies that $\overline{T}\,=\,(\mathbb{P}^{1})^{n},$ and $
\rm{Lie}({\rm Aut}^0(\overline{T})) = \rm{sl}(2, {\mathbb C})^n,$ where $n\,=\,\dim T.$ The $T$-fixed points
of $(\mathbb{P}^{1})^{n}$ are indexed by the elements of the Weyl group of
${\rm PSL}(2, {\mathbb C})^n$.
Therefore, $\overline{T}$ has $2^{n}$ fixed points for the action
of $T.$ On the other hand, by \cite[p. 477, (1.2.7) and p. 478, (1.3.8)]{BJ},
all $w(z)\,\in\, \overline{T}$, $w\,\in\, W$, are fixed by $T$,
and $w'(z) \,=\, w(z)$ only if $w'\,= \,w.$ Consequently, the order of $W$
is at most $2^{n}$.
As $n\,=\,\dim T$, this is possible only if $W=S_{2}^{n}.$
Hence it follows that $G\,=\,{\rm PSL}(2,{\mathbb C})^{n}$. But this contradicts
the assumption that $G$ is simple of rank $n\,\geq\, 2$. So ${\rm Aut}^{0}(\overline{T})$ is not semisimple.

The group ${\rm Aut}^{0}(\overline{T})$ is reductive but not semisimple, and this
implies that the connected
component $Z^0$, containing the identity element, of the center of ${\rm Aut}^{0}(
\overline{T})$ is a positive dimensional sub-torus of $T.$ Further, since
$$
w{\rm Aut}^{0} (\overline{T})w^{-1}\,=\,{\rm Aut}^{0}(\overline{T})\, ,
$$ it follows that $wZ^{0}w^{-1}\,=\,Z^{0}$ for every $w\,\in\, W.$
Thus, the restriction map
\begin{equation}\label{hom}
r\,:\, X(T)\otimes_{\mathbb Z} \mathbb{R}\,\longrightarrow \,X(Z^{0})
\otimes_{\mathbb Z} \mathbb{R}
\end{equation}
 is a nonzero
homomorphism of $W$ modules. Note that $X(T)\otimes \mathbb{R}$ is an irreducible $W$ module
(this is because $G$ is simple). So we conclude that the homomorphism $r$ in
\eqref{hom} is an isomorphism. Consequently, we have $T\,=\,Z^{0}$ and 
 $T\,=\,{\rm Aut}^0(\overline{T}).$ 
\end{proof}

\section*{Acknowledgements}

We are grateful to the referee for helpful comments. The first--named 
author thanks the Institute of Mathematical Sciences for hospitality while this work 
was carried out. He also acknowledges the support of the J. C. Bose Fellowship. The 
second named author would like to thank the Infosys Foundation for the partial support.

\end{document}